\documentclass[12pt]{article}
\usepackage{amsfonts}
\usepackage{latexsym,amsmath,amssymb,amsfonts,epsfig,psfrag,url,graphics,ifpdf,multicol}
\usepackage{eepic,color,colordvi,amscd,amsthm}
\usepackage[section]{algorithm}
\usepackage{algpseudocode}
\usepackage[numbers,sort&compress]{natbib}
\usepackage{indentfirst,graphics,epsfig}
\input{epsf}
 \usepackage[dvipdfmx]{hyperref}%%
\hypersetup{pdfpagemode=UseOutlines,
bookmarksopen,%pdfpagelabels,plainpages=false,
pdfauthor=Grace B. Huo, pdfstartview=FitH, breaklinks=true,
colorlinks=true,linkcolor=blue,citecolor=blue,urlcolor=blue,anchorcolor=black}

\newtheorem{theo}{Theorem}[section]
\newtheorem{lemm}[theo]{Lemma}

\newtheorem{conj}[theo]{Conjecture}

\theoremstyle{definition}

\theoremstyle{remark}

\usepackage{indentfirst,subfig}

\begin{document}

\begin{center} {\Large \bf Complete Solution to a Problem on the Maximal
Energy of Unicyclic Bipartite Graphs \footnote{Supported by NSFC and
``the Fundamental Research Funds for the Central Universities". }}

\end{center}

\begin{center}
{
  {\small  Bofeng Huo$^{1,2}$, Xueliang Li$^1$, Yongtang Shi$^1$} \\[3mm]
  {\small $^1$Center for Combinatorics and LPMC-TJKLC}\\
  {\small Nankai University, Tianjin 300071, China}\\
  {\small E-mail: huobofeng@mail.nankai.edu.cn;  lxl@nankai.edu.cn; shi@nankai.edu.cn}\\
  {\small $^2$Department of Mathematics and Information Science}\\
  {\small Qinghai Normal University, Xining 810008, China}\\[2mm]

}
\end{center}

\begin{center}
\begin{minipage}{120mm}
\begin{center}
{\bf Abstract}
\end{center}
{\small The energy of a simple graph $G$, denoted by $E(G)$, is
defined as the sum of the absolute values of all eigenvalues of its
adjacency matrix. Denote by $C_n$ the cycle, and $P_n^{6}$ the
unicyclic graph obtained by connecting a vertex of $C_6$ with a leaf
of $P_{n-6}$\,. Caporossi et al. conjecture that the unicyclic graph
with maximal energy is $P_n^6$ for $n=8,12,14$ and $n\geq 16$.
In``Y. Hou, I. Gutman and C. Woo, Unicyclic graphs with maximal
energy, {\it Linear Algebra Appl.} {\bf 356}(2002), 27--36", the
authors proved that $E(P_n^6)$ is maximal within the class of the
unicyclic bipartite $n$-vertex graphs differing from $C_n$\,. And
they also claimed that the energy of $C_n$ and $P_n^6$ is
quasi-order incomparable and left this as an open problem. In this
paper, by utilizing the Coulson integral formula and some knowledge
of real analysis, especially by employing certain combinatorial
techniques, we show that the energy of $P_n^6$ is greater than that
of $C_n$ for $n=8,12,14$ and $n\geq 16$, which completely solves
this open problem and partially solves the above conjecture.
\\ [2mm]
}

\end{minipage}
\end{center}

\section{Introduction}

Let $G$ be a simple graph of order $n$, $A(G)$ the
adjacency matrix of $G$. The characteristic polynomial of $A(G)$ is usually
called the characteristic polynomial of $G$, denoted by
\begin{equation*}\label{chapoly1}
\phi(G, x)=\det(xI-A(G))=x^n+a_1x^{n-1}+\cdots+a_n,
\end{equation*}
It is well-known \cite{cvet1980} that the characteristic polynomial
of a bipartite graph $G$ takes the form
\begin{equation*}\label{cvetgraph}
\phi(G,
x)=\sum_{k=0}^{\lfloor{n/2}\rfloor}a_{2k}x^{n-2k}=\sum_{k=0}^{\lfloor{n/2}\rfloor}(-1)^{k}b_{2k}x^{n-2k},
\end{equation*}
where $b_{2k}=(-1)^ka_{2k}$ and $b_{2k}\geq 0$ for all
$k=1,\ldots,\lfloor{n/2}\rfloor$, especially $b_0=a_0=1$. Moreover,
the characteristic polynomial of a tree $T$ can be expressed as
\begin{equation*}\label{cvettree}
\phi(T,  x)=\sum_{k=0}^{\lfloor{n/ 2}\rfloor}(-1)^{k}m(T,k)x^{n-2k},
\end{equation*}
where $m(T,k)$ is the number of $k$-matchings of $T$.

For a graph $G$, Let $\lambda_1, \lambda_2,\ldots,
\lambda_n$ denote the eigenvalues of its characteristic polynomial. The
energy of a graph $G$ is defined as
$$E(G)=\sum_{i=1}^n|\lambda_i|.$$
This definition was proposed by Gutman \cite{gutman1978}. The
following formula is also well-known
\begin{equation*}
E(G)={1\over\pi}\int^{+\infty}_{-\infty}{1\over x^2}\log |x^n
{\phi(G,i/x)}|\mathrm{d}x,
\end{equation*}
where $i^2=-1$. Furthermore, in the book of Gutman and Polansky
\cite{gutman&polansky1986}, the above equality was converted into an
explicit formula as follows:
\begin{equation*}\label{tujifen}
E(G)={1\over2\pi}\int^{+\infty}_{-\infty}{1\over
x^2}\log\left[\left(\sum_{k=0}^{\lfloor{n/
2}\rfloor}(-1)^ka_{2k}x^{2k}\right)^2+\left(\sum_{k=0}^{\lfloor{n/
2}\rfloor}(-1)^ka_{2k+1}x^{2k+1}\right)^2\right] \mathrm{d}x.
\end{equation*}
For more results about graph energy, we refer the reader to the
recent survey of Gutman, Li and Zhang \cite{gutman&lxl2009}.

For two trees $T_1$ and $T_2$ of the same order, one can introduce a
quasi order $\preceq$ in the set of trees, namely, if $m(T_1,k)\leq
m(T_2,k)$ holds for all $k\geq 0$, then define $T_1\preceq T_2$, and
so $T_1\preceq T_2$ implies $E(T_1)\leq E(T_2)$ (e.g.
\cite{gutman1977}). Similarly, one can generalize the quasi order to
the cases of bipartite graphs (e.g. \cite{lxl&zjb2009}) and
unicyclic graphs (e.g. \cite{hou2001}). The quasi order method is
commonly used to compare the energies of two trees, bipartite graphs
or unicyclic graphs. However, for general graphs, it is difficult to
define such a quasi order. If, for two trees or bipartite graphs,
the above quantities $m(T, k)$ or $|a_k(G)|$ can not be compared
uniformly, then the common comparing method is invalid, and this
happened very occasionally. Recently, for these quasi-order
incomparable problems, we find an efficient way to determine which
one attains the extremal value of the energy, see \cite{Huo&Ji&Li1,
Huo&Ji&Li2, Huo&Ji&Li3, HJLS20101028}.

Let $C_n$ be the cycle, and $P_n^{6}$ be the unicyclic graph
obtained by connecting a vertex of $C_6$ with a leaf of $P_{n-6}$\,.
In \cite{CC}, Caporossi et al. proposed a conjecture on the
unicyclic graph with the maximum energy.
\begin{conj}
Among all unicyclic graphs on $n$ vertices, the cycle $C_n$ has
maximal energy if $n\leq 7$ and $n=9,10,11,13$ and $15$\,. For all
other values of $n$\,, the unicyclic graph with maximal energy is
$P_n^6$\,.
\end{conj}

\begin{theo}\label{HLSthm0}
Let $G$ be any connected, unicyclic and bipartite graph on $n$
vertices and $G \ncong  C_n$\,. Then ${E}(G)<{E}(P_n^6)$\,.
\end{theo}

In \cite{hou2002}, the authors proved Theorem \ref{HLSthm0} that is
weaker than the above conjecture, namely that $E(P_n^6)$ is maximal
within the class of the unicyclic bipartite $n$-vertex graphs
differing from $C_n$\,. And they also claimed that the energy of
$C_n$ and $P_n^6$ is quasi-order incomparable. In this paper, we
will employ the Coulson integral formula and some knowledge of
analysis, especially by using certain combinatorial techniques, to
show that $E(C_n)<E(P_n^6)$,  and then completely determine that
$P_n^6$ is the only graph which attains the maximum value of the
energy among all the unicyclic bipartite graphs, which partially
solves the above conjecture.

\begin{theo}\label{HJLStheo}
For $n=8,12,14$ and $n\geq 16$, $E(P_n^6)>E(C_n)$.
\end{theo}

\section{Main results}
We recall some knowledge on real analysis, for which we refer to
\cite{zorich2002}.

\begin{lemm}\label{inequalitylemm}
For any real number $X>-1$, we have
 \begin{equation*}
\frac{X}{1+X}\leq\log(1+ X)\leq X.
\end{equation*}
\end{lemm}
The following lemma is a well-known result due to Gutman
\cite{gutman2001}, which will be used in the sequel.
\begin{lemm}\label{lemma1.1}
If $G_1$ and $G_2$ are two graphs with the same number of vertices,
then
$$E(G_1)-E(G_2)={1\over\pi}\int^{+\infty}_{-\infty}
\log\left|\frac{ \phi (G_1, ix)}{\phi(G_2, ix)}\right|
\mathrm{d}x.$$
\end{lemm}
In the following, we list some basic properties of the
characteristic polynomial $\phi(G,x)$\,, which can be found in
\cite{cvet1980}.

\begin{lemm}
\label{zh8} Let $uv$ be an edge of $G$\,. Then
$$
\phi (G,x)=\phi (G-uv,x) - \phi (G-u-v,x)-2\sum _{C\in {\mathcal
C}(uv)}\phi (G-C,x)
$$
where ${\mathcal C}(uv)$ is the set of cycles containing $uv$\,. In
particular, if $uv$ is a pendent edge with pendent vertex $v$\,,
then $ \phi (G,x)=x\,\phi (G-v,x)-\phi (G-u-v,x)$\,.
\end{lemm}

Now we can easily obtain the following lemma from Lemma \ref{zh8}.

\begin{lemm}\label{HJLSlem1}
$\phi(P_n^6,x)=x\phi(P_{n-1}^6,x)-\phi(P_{n-2}^6,x)$ and $\phi
(C_n,x)=\phi (P_{n},x)-\phi (P_{n-2},x)-2$.
\end{lemm}

By some easy calculations, we get
$\phi(P_8^6,x)=x^8-8x^6+19x^4-16x^2+4$ and
$\phi(P_7^6,x)=x^7-7x^5+13x^3-7x$. Now for convenience, we introduce
some notions as follows
\begin{align*}
&Y_1(x)=\frac{x+\sqrt{x^2-4}}{2},\qquad\qquad\qquad~ Y_2(x)=\frac{x-\sqrt{x^2-4}}{2},\\
&C_1(x)=\frac {Y_1(x)(x^2-1)-x}
{(Y_1(x))^{3}-Y_1(x)},\quad\quad~~~~~~ C_2(x)=\frac {Y_2(x)(x^2-1)-x} {(Y_2(x))^{3}-Y_2(x)},\\
&A_1(x)=\frac {Y_1(x)\phi (P_8^6,x)-\phi (P_7^6,x)}
{(Y_1(x))^{9}-(Y_1(x))^{7}},~~ A_2(x)=\frac {Y_2(x)\phi
(P_8^6,x)-\phi (P_7^6,x)} {(Y_2(x))^{9}-(Y_2(x))^{7}}.
\end{align*}
It is easy to verify that $Y_1(x)+Y_2(x)=x$, $Y_1(x)Y_2(x)=1$,
$Y_1(ix)=\frac{x+\sqrt{x^2+4}}{2}i$ and
$Y_2(ix)=\frac{x-\sqrt{x^2+4}}{2}i$. We define
$$f_8=x^8+8x^6+19x^4+16x^2+4,~~f_7=x^7+7x^5+13x^3+7x$$ and
$$Z_1(x)=-iY_1(ix)=\frac{x+\sqrt{x^2+4}}{2},~
Z_2(x)=-iY_2(ix)=\frac{x-\sqrt{x^2+4}}{2}.$$

\begin{lemm}\label{HJLSlem0}
For $n\geq 10$ and $x\neq \pm 2$, the characteristic polynomials of
$P_n^6$ and $C_n$ have the following form
$$\phi (P_n^6,x)=A_1(x)(Y_1(x))^n+A_2(x)(Y_2(x))^n$$ and $$\phi (C_n,x)=(Y_1(x))^n+(Y_2(x))^n-2.$$
\end{lemm}
\begin{proof}
By Lemma \ref{HJLSlem1}, we notice that $\phi (P_n^6,x)$ satisfy the
recursive formula $f(n,x)=xf(n-1,x)-f(n-2,x)$. Therefore, the
general solution of this linear homogeneous recurrence relation is
$f(n,x)=D_1(x)(Y_1(x))^n+D_2(x)(Y_2(x))^n$. By some elementary
calculations,  we can easily obtain that $D_i(x)=A_i(x)$ for $\phi
(P_n^6,x)$, $i=1,2$, from the initial values $\phi(P_8^6,x)$,
$\phi(P_7^6,x)$.

By Lemma \ref{HJLSlem1}, $\phi (C_n,x)=\phi (P_{n},x)-\phi
(P_{n-2},x)-2$ and $\phi (P_n,x)$ satisfy the recursive formula
$f(n,x)=f(n-1,x)-f(n-2,x)-3$. Similarly, we can obtain the general
solution of this linear nonhomogeneous recurrence relation from the
initial values $\phi(P_1,x)=x$, $\phi(P_2,x)=x^2-1$.
\end{proof}

{\bf Proof of Theorem \ref{HJLStheo}}~~ For $n=8,12,14$, it is easy
to verify $E(P_n^6)>E(C_n)$. In the following, we always suppose
$n\geq 16$. Using Lemma \ref{lemma1.1}, we can deduce
\begin{align*} E(C_n)-E(P_n^6)=\frac 1 \pi
\int_{-\infty}^{+\infty}\log\left|\frac{ \phi (C_n, ix)}{\phi(P_n^6,
ix)}\right| \mathrm{d}x.
\end{align*}
From Lemma \ref{HJLSlem0}, we have
\begin{align*}
&\phi(C_n,ix)=(Y_1(ix))^n+(Y_2(ix))^n-2=((Z_2(x))^2(x^2+1)-(Z_2(x))^3x)(Z_1(x))^n\cdot i^n\\[2mm]
&\qquad\qquad\quad+((Z_1(x))^2(x^2+1)-(Z_1(x))^3x)(Z_2(x))^n\cdot
i^n-2,\\[2mm]
&\phi(P_n^6,ix)=A_1(ix)(Y_1(ix))^n+A_2(ix)(Y_2(ix))^n\\[2mm]
&\qquad\qquad=\frac{Z_1(x)f_8+f_7}{(Z_1(x))^9+Z_1^7}(Z_1(x))^n\cdot
i^n+\frac{Z_2(x)f_8+f_7}{(Z_2(x))^9+(Z_2(x))^7} (Z_2(x))^n\cdot i^n.
\end{align*}

Firstly, we will prove that $E(C_n)-E(P_n^6)$ is an decreasing
function of $n$ for $n=4k+j$, $j=1,2,3$, namely,
\begin{align*}
&\log\left|\frac{(Y_1(ix))^{n+4}+(Y_2(ix))^{n+4}-2}{A_1(ix)(Y_1(ix))^{n+4}+A_2(ix)(Y_2(ix))^{n+4}}\right|
-\log\left|\frac{(Y_1(ix))^{n}+(Y_2(ix))^{n}-2}{A_1(ix)(Y_1(ix))^{n}+A_2(ix)(Y_2(ix))^{n}}\right|\\[2mm]
=& \log\left(1+\frac{K_0(n,x)}{H_0(n,x)}\right)<0.
\end{align*}

{\bf Case 1} $n=4k+2$.

In this case, $H_0(n,x)=\left|\phi (C_n, ix)\cdot \phi(P_{n+4}^6,
ix)\right|>0$ and
\begin{align*}
K_0(n,x)=&\left(A_1(ix)-A_2(ix)\right)
\left((Y_2(ix))^{4}-(Y_1(ix))^{4}\right)-2A_1(ix)(Y_1(ix))^{n}(1-(Y_1(ix))^{4})\\
&-2A_2(ix)(Y_2(ix))^{n}(1-(Y_2(ix))^{4}).
\end{align*}
Then, by some elementary calculations, we have
\begin{align*}
K_0(n,x)=x(x^2+1)&(x^9+9x^7+30x^5+46x^3+28x \\
&~+(Z_2(x))^n(x^5+5x^3+6x+\sqrt{x^2+4}\,(x^4+3x^2+4))\\
&~+(Z_1(x))^n(x^5+5x^3+6x-\sqrt{x^2+4}\,(x^4+3x^2+4))).
\end{align*}
If $x>0$, then $Z_1(x)> 1$, $-1< Z_2(x)< 0$, and we obtain
\begin{align*}
K_0(n,x)=x(x^2+1)(Z_1(x))^nq(n,x) < x(x^2+1)(Z_1(x))^nq(10,x),
\end{align*}
where
\begin{align*}
q(n,x)=&(Z_2(x))^n(x^9+9x^7+30x^5+46x^3+28x)\\
&+(Z_2(x))^{2n}(x^5+5x^3+6x+\sqrt{x^2+4}\,(x^4+3x^2+4))\\
&+x^5+5x^3+6x-\sqrt{x^2+4}\,(x^4+3x^2+4).
\end{align*}
By some simplifications,
\begin{align*}
q(10,x)&=-\frac 1
2x(x^2+4)(2x^8+17x^6+47x^4+46x^2+10)\,\cdot\\
&(x^{10}+10x^8+35x^6+50x^4+25x^2+2-\sqrt{x^2+4}(x^9+8x^7+21x^5+20x^3+5x)).
\end{align*}
Since
$$\left(x^{10}+10x^8+35x^6+50x^4+25x^2+2\right)^2-\left(\sqrt{x^2+4}\,(x^9+8x^7+21x^5+20x^3+5x)\right)^2=4,$$
we have $q(10,x)<0$, and hence $\frac{K_0(n,x)}{H_0(n,x)}<0$.
Similarly, we can prove $\frac{K_0(n,x)}{H_0(n,x)}<0$ for $x<0$.

Therefore, we have shown that $E(C_n)-E(P_n^6)$ is an decreasing
function of $n$ for $n=4k+2$.

{\bf Case 2} $n=4k+j$, $j=1,3$.

In this case, $H_0(n,x)=\left|\phi (C_n, ix)\cdot \phi(P_{n+4}^6,
ix)\right|>0$ and
\begin{align*}
K_0(n,x)=&\left|((Y_1(ix))^{n+4}+(Y_2(ix))^{n+4}-2)(A_1(ix)(Y_1(ix))^{n}+A_2(ix)(Y_2(ix))^{n})\right|\\
&-
\left|(A_1(ix)(Y_1(ix))^{n+4}+A_2(ix)(Y_2(ix))^{n+4})((Y_1(ix))^{n}
+(Y_2(ix))^{n}-2)\right|\\
=&\sqrt{p(n,x)}-\sqrt{w(n,x)}, \end{align*} where
\begin{align*}
p(n,x)=&\left(A_2(ix)(Z_1(x))^4+A_1(ix)(Z_2(x))^4-A_1(ix)(Z_1(x))^{2n+4}-A_2(ix)(Z_2(x))^{2n+4}\right)^2\\
&+\left(-2A_1(ix)(Z_1(x))^n-2A_2(ix)(Z_2(x))^n\right)^2,\\
w(n,x)=&\left(A_1(ix)(Z_1(x))^4+A_2(ix)(Z_2(x))^4-A_1(ix)(Z_1(x))^{2n+4}-A_2(ix)(Z_2(x))^{2n+4}\right)^2\\
&+\left(-2A_1(ix)(Z_1(x))^{n+4}-2A_2(ix)(Z_2(x))^{n+4}\right)^2.
\end{align*}
Now we only need to check $p(n,x)-w(n,x)<0$ for all $x$ and $n$.
Firstly we suppose $n=4k+1$. If $x>0$, then
$(Z_1(x))^{2n}>(Z_1(x))^{10}$, $(Z_2(x))^{2n}<(Z_2(x))^{10}$, and we
have
\begin{align*}
p(n,x)-w(n,x)=&~x(x^2+2)^3(x^2+1)^3(x^{11}+11x^9+46x^7+92x^5+88x^3+28x\\
&\qquad\qquad\quad\qquad\qquad-2(Z_1(x))^{2n}(\sqrt{x^2+4}\,(x^2+2)+x)\\
&\qquad\qquad\quad\qquad\qquad+2(Z_2(x))^{2n}(\sqrt{x^2+4}\,(x^2+2)-x))\\
<&~p(5,x)-w(5,x)\\
=&~-x^2(x^2+4)(x^2+1)^4(x^2+2)^3(2x^8+19x^6+60x^4+68x^2+14)<0.
\end{align*}
If $x<0$, then $(Z_1(x))^{2n}<(Z_1(x))^{10}$,
$(Z_2(x))^{2n}>(Z_2(x))^{10}$. Similarly,
$p(n,x)-w(n,x)<p(5,x)-w(5,x)<0$. By the same discussion as the case
of $n=4k+1$, for $n=4k+3$ and $x>0$ or $x<0$, we can deduce that
\begin{align*}
&p(n,x)-w(n,x)<p(7,x)-w(7,x)\\
=&-x^2(x^2+4)(x^2+2)^3(x^2+1)^3\,\cdot\\
&(2x^{14}+30x^{12}+178x^{10}+533x^8+849x^6+690x^4+242x^2+22)<0.
\end{align*}
Thus, we have done for $n=4k+j$, $j=1,3$.

Therefore, we have shown that $E(C_n)-E(P_n^6)$ is an decreasing
function of $n$ for $n=4k+j$, $j=1,2,3$. So, when $n=4k+2$,
$E(C_n)-E(P_n^6)<E(C_{18})-E(P_{18}^6)\doteq -0.03752<0$; when
$n=4k+1$, $E(C_n)-E(P_n^6)<E(C_{17})-E(P_{17}^6)\doteq -0.00961<0$;
when $n=4k+3$, $E(C_n)-E(P_n^6)<E(C_{19})-E(P_{19}^6)\doteq
-0.02290<0$.

Finally, we will deal with the case of $n=4k$. Notice that in this
case both $\phi (C_n, ix)$ and $\phi(P_{n}^6, ix)$ are polynomials
of $x$ with all real coefficients. When $n\rightarrow \infty$,
$$\frac{(Y_1(ix))^n+(Y_2(ix))^n-2}{A_1(ix)(Y_1(ix))^n+A_2(ix)(Y_2(ix))^n}\rightarrow\left\{
 \begin{array}{ll}
 \frac{1}{A_1(ix)} &\mbox{if $x>0$}\\[3mm]
\frac{1}{A_2(ix)} &\mbox{if $x<0$}.
 \end{array}
 \right.
 $$
In this case, we will prove $$\log
\frac{(Y_1(ix))^n+(Y_2(ix))^n-2}{A_1(ix)(Y_1(ix))^n+A_2(ix)(Y_2(ix))^n}<\log
\frac{1}{A_1(ix)}$$ for $x>0$ and  $$\log
\frac{(Y_1(ix))^n+(Y_2(ix))^n-2}{A_1(ix)(Y_1(ix))^n+A_2(ix)(Y_2(ix))^n}<\log
\frac{1}{A_2(ix)}$$ for $x<0$. In the following we only check the
case of $x>0$ as the case of $x<0$ is similarly. Assume
$$\log
\frac{(Y_1(ix))^n+(Y_2(ix))^n-2}{A_1(ix)(Y_1(ix))^n+A_2(ix)(Y_2(ix))^n}-\log
\frac{1}{A_1(ix)}=\log\left(1+\frac {K_1(n,x)} {H_1(n,x)}\right),$$
by some elementary calculations, we obtain $H_1(n,x)>0$ and
\begin{align*}
&K_1(n,x)=-\frac{x^2+1}{x^2+4}\cdot\\
&\left(x^8+9x^6+28x^4+36x^2+16+((Z_2(x))^n-1)\sqrt{x^2+4}\,(x^7+7x^5+16x^3+14x)\right)\\
&<-\frac{x^2+1}{x^2+4}\cdot\left(x^8+9x^6+28x^4+36x^2+16-\sqrt{x^2+4}\,(x^7+7x^5+16x^3+14x)\right)<0,
\end{align*}
since
\begin{align*}
&\left(x^8+9x^6+28x^4+36x^2+16\right)^2-\left(\sqrt{x^2+4}\,(x^7+7x^5+16x^3+14x\right)^2\\
=&~4x^8+48x^6+204x^4+368x^2+256>0.
\end{align*}
Notice that if $x>0$, then
$A_1(ix)=\frac{Z_1(x)f_8+f_7}{(Z_1(x))^9+Z_1^7}>0$, and if $x<0$,
then $A_2(ix)=\frac{Z_2(x)f_8+f_7}{(Z_2(x))^9+Z_2^7}=
\frac{Z_1(x)\cdot\left(Z_2(x)f_8+f_7\right)}{Z_1(x)\cdot\left((Z_2(x))^9+Z_2^7\right)}
=\frac{f_8-Z_1(x)f_7}{(Z_2(x))^8+Z_2^6}>0$.
Thus, by Lemma
\ref{inequalitylemm}, we have $$\frac 1 \pi\int_{0}^{+\infty}\log
\frac 1 {A_1(ix)}\mathrm{d}x<\frac 1
\pi\int_{0}^{+\infty}\left(\frac 1
{A_1(ix)}-1\right)\mathrm{d}x\doteq -0.047643;$$
$$\frac 1 \pi\int_{-\infty}^{0}\log \frac 1
{A_2(ix)}\mathrm{d}x<\frac 1 \pi\int_{0}^{+\infty}\left(\frac 1
{A_2(ix)}-1\right)\mathrm{d}x\doteq -0.047643.$$ Therefore,
$$E(C_n)-E(P_n^6)<\frac 1
\pi\int_{0}^{+\infty}\log \frac 1
{A_1(ix)}\mathrm{d}x+\int_{-\infty}^{0}\log \frac 1
{A_2(ix)}\mathrm{d}x<-0.047643-0.047643<0. $$

The proof is completed. \qed

\end{document}